\documentclass[12pt]{article}
\usepackage{amsmath,amsthm}
\usepackage{amssymb}
\usepackage{epsfig}
\usepackage{tikz}
\usepackage{graphicx}
\usepackage{url}
\usepackage{authblk}
\theoremstyle{plain}
\newtheorem{theorem}{Theorem}[section]
\newtheorem{corollary}[theorem]{Corollary}
\newtheorem{lemma}[theorem]{Lemma}
\newtheorem{prop}[theorem]{Proposition}
\theoremstyle{definition}

\newtheorem{remark}[theorem]{Remark}
\newtheorem{example}[theorem]{Example}

\newtheorem{problem}[theorem]{Problem}

\newcommand{\Aut}{\mathop{\mathrm{Aut}}}

\begin{document}
  \setcounter{Maxaffil}{3}
   \title{Forbidden subgraphs of power graphs}
   \author[a]{Pallabi Manna\thanks{mannapallabimath001@gmail.com}}
   \author[b]{Peter J. Cameron\thanks{pjc20@st-andrews.ac.uk}}
   \author[a]{Ranjit Mehatari\thanks{ranjitmehatari@gmail.com, mehatarir@nitrkl.ac.in}}
   \affil[a]{Department of Mathematics,}
   \affil[ ]{National Institute of Technology Rourkela,}
   \affil[ ]{Rourkela - 769008, India\\}
    \affil[ ]{ }
   \affil[b]{School of Mathematics and Statistics,}
   \affil[ ]{University of St Andrews,}
   \affil[ ]{North Haugh, St Andrews, Fife, KY16 9SS, UK}

   \maketitle
\begin{abstract}
The undirected power graph (or simply power graph) of a group $G$, denoted
by $P(G)$, is a graph whose vertices are the elements of the group $G$, in
which two vertices $u$ and $v$ are connected by an edge between if and only
if either $u=v^i$ or $v=u^j$ for some $i$, $j$.

A number of important graph classes, including perfect graphs, cographs,
chordal graphs, split graphs, and threshold graphs, can be defined either
structurally or in terms
of forbidden induced subgraphs. We examine each of these five classes and
attempt to determine for which groups $G$ the power graph $P(G)$ lies in the
class under consideration. We give complete results in the case of nilpotent
groups, and partial results in greater generality. In particular, the
power graph is always perfect; and we determine completely the groups whose
power graph is a threshold or split graph (the answer is the same for both
classes). We give a number of open problems.
\end{abstract}

\section{Introduction}

The study of graph representations is one of the interesting and popular 
research topic in algebraic graph theory. One of the major graph
representation amongst them is the power graphs of finite groups. We found
several papers in this context
\cite{Abawajy,Cameron,Ghosh,Chakrabarty,Chattopadhyay,Curtin,dostabadi,Kelarev,Ma1,Ma,Pourghobadi}.
\begin{example}\label{order8}
There are five groups of order $8$. The power graph of $C_8$ is complete,
and the power graph of $(C_2)^3$ is a star $K_{1,7}$. The figure shows
the power graphs of the other three groups.
\begin{center}
\begin{tabular}{ccc}
\includegraphics[scale=0.7]{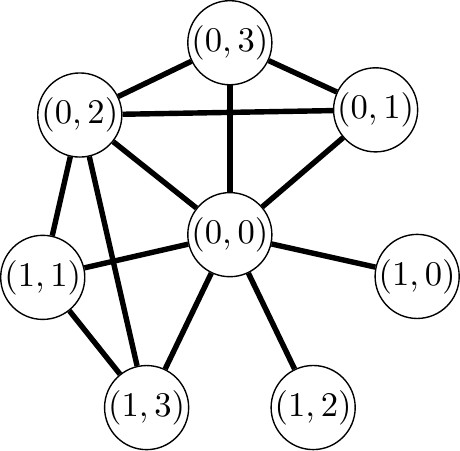} & \includegraphics[scale=1.0]{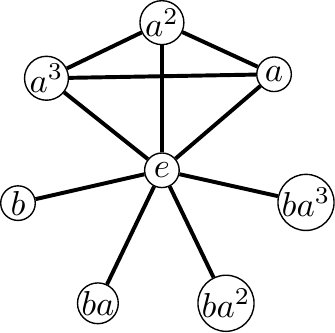}
 & \includegraphics[scale=1.0]{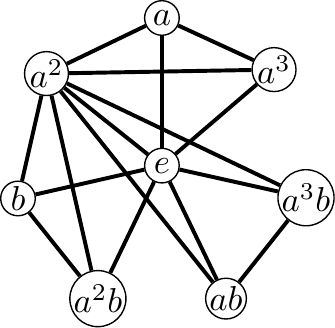} \\
$C_4\times C_2$ & $D_4$ & $Q_8$ \\
\end{tabular}
\end{center}
\end{example}

We begin with the history of the concept.
The notion of power graphs was first introduced by Kelarev and Quinn
\cite{Kelarev} in 2002.
For a semigroup $S$, the \emph{directed power graph} of $S$, denoted by
$ \vec{P}(S)$,  is a graph with vertex set $V=S$; for two distinct
vertices $u,v\in S$, there exists an arc $(u,v)$
if $v$ is a power of~$u$. 

The corresponding undirected graph is called  the
\emph{undirected power graph} of $S$, denoted by $P(S)$. The undirected
power graph of a semigroup was introduced by
Chakrabarty et~al.~\cite{Chakrabarty} in 2009. So the undirected power graph of
$S$ is the graph with vertex set $V(P(S))=S$,
with an edge between two vertices $u$ and $v$ if $ u\ne v$ and either $v$ is
a power of $u$ or $u$ is a power of~$v$. In the sequel, ``power graph'' will
mean ``undirected power graph''.

Later the power graphs of groups were studied. In \cite{Chakrabarty}, the
authors proved that if $G$ is a finite group then the power graph $P(G)$ is
always connected. They further showed that a finite group has a complete
undirected power graph if and only if it is cyclic and its order is equal to
$1$ or $p^{m}$ for some prime $p$. (We give the proof below.) They also counted
the number of edges in a power graph of a finite group $G$ by the formula
$|E(P(G))|=\dfrac{1}{2}[\sum_{g\in G}(2o(g)-\phi(o(g))-1)]$. The power graph
of a finite group is Eulerian if and only if the group $G$ is of odd order
\cite{Chakrabarty}. Also $P(G)$ is Hamiltonian if $|G|\geq 3$.

In \cite{Curtin}, Curtin \emph{et~al.}\ introduced the concept of proper power
graphs. The \emph{proper power graph} of a group $G$, denoted by $P^*(G)$, is
the graph obtained from $P(G)$ by deleting the identity. They discussed
the diameter of the proper  power graph of the symmetric group $S_n$ on
$n$ symbols. For more information related to proper power graphs we refer to
\cite{Curtin, dostabadi,Pourghobadi}.

In \cite{Chattopadhyay}, Chattopadhyay \emph{et~al.} gave bounds for the
vertex connectivity of the power graph of the cyclic group $C_n$.

In \cite{Ghosh}, Cameron and Ghosh
investigated isomorphism of the power graphs of groups. They showed that if
two finite Abelian groups have isomorphic power graphs then they are isomorphic.
Moreover they proved that if two finite groups have isomorphic directed power
graphs then they have same number of elements of each order. They have also
proved that the only finite group $G$ for which $\Aut(G)=\Aut(P(G))$ is the
Klein $4$-group $V_4=C_2\times C_2$. In \cite{Cameron}, Cameron proved
that if two finite groups have isomorphic power graphs then they have
isomorphic directed power graphs.

Papers dealing with the power graphs of infinite groups include
\cite{aacns,cgj,cj,jafari,z1,z2}.

For more information related to power
graphs we refer to the survey article by Abawajy \emph{et~al.}~\cite{Abawajy}.

\medskip

Let $G=(V,E)$  be any graph and $S$ be any subset of $V$.
Then the induced subgraph $G[S]$ is the graph whose vertex set is $S$ and
whose edge set consists of all of the edges in $E$ that have both endpoints
in $S$. For any graph $H$, the graph $G$ is said to be \emph{$H$-free} if it
has no induced subgraph isomorphic to $H$. The graphs $P_n$, $C_n$ and $2K_2$
denote the path on $n$ vertices, the cycle with $n$ vertices, and two disjoint
edges with no further edges connecting them.

We use $C_n$ for both the cyclic group of order $n$ and the cycle graph on
$n$ vertices; the context makes clear in each case which is intended. Also,
we use $D_n$ for the dihedral group
\[\langle a,b:a^n=b^2=(ab)^2=1\rangle\]
of order $2n$, rather than for the dihedral group of order $n$ as is often
done.

\medskip

Our general theme is that various important classes of graphs are defined
by forbidden induced subgraphs \cite{Brand,Trotignon}, and we investigate several of these classes
with the aim of determining for which groups $G$ the power graph $P(G)$
belongs to the corresponding class. So we conclude the Introduction with a
simple result of this form due to Chakrabarty
\emph{et~al.} \cite{Chakrabarty}, which illustrates the general
problem and will also be used later. Note that a graph is complete if and only
if it does not contain the null graph on $2$ vertices as an induced subgraph.

\begin{prop}\label{p:complete}
The power graph of a finite group $G$ is complete if and only if $G$ is a
cyclic group of prime power order.
\end{prop}

\begin{proof}
In a cyclic group $G$, $a$ is a power of $b$ if and only if the order of $a$
divides the order of $b$. If $|G|$ is a power of a prime $p$, then the
possible orders are powers of $p$, and so are totally ordered; so $P(G)$
is complete.

Conversely, suppose that $P(G)$ is complete. Then $|G|$ must be a prime power;
for, if distinct primes $p$ and $q$ divide $|G|$, then $G$ contains elements
$a$ and $b$ of orders $p$ and $q$ respectively, and these are nonadjacent in
$P(G)$.

Let $|G|=p^n$. Then, for each $i$ with $i\le n$, the number of elements 
of order $p^i$ is at most $\phi(p^i)$, where $\phi$ is Euler's
function, since if $a$ and $b$ are two such elements then
$b\in\langle a\rangle$, and a cyclic group of order $p^i$ contains $\phi(p^i)$
elements of order $p^i$. Since
\[\sum_{i\le n}\phi(p^i)=p^n,\]
we see that there are $\phi(p^i)$ elements of order $p^i$ for all $i$. In
particular, there exist elements of order $p^n$; so $G$ is cyclic.
\end{proof}

\section{The power graph is perfect}

A finite graph $\Gamma$ is \emph{perfect} if it has the property that every
induced subgraph has clique number equal to chromatic number. We are going
to show that every power graph is perfect. This motivates considering
subclasses of the class of perfect graphs, as we do in this paper.

The class of perfect graphs is closed under complementation (this is the
\emph{Weak Perfect Graph Theorem} of Lov\'asz~\cite{wpgt}), and contains
several further important graph classes, including bipartite graphs,
line graphs, chordal graphs, interval graphs, and comparability graphs of
partial orders. It was shown by Gr\"otschel, Lov\'asz and Schrijver \cite{gls}
that computational problems which are hard for general graphs (such as graph
colouring, maximal clique and maximal independent set) can be solved in
polynomial time for perfect graphs using semidefinite programming.

According to the \emph{Strong Perfect Graph Theorem}, a conjecture of 
Berge~\cite{berge} proved by Chudnovsky,
Robertson, Seymour and Thomas~\cite{spgt}, a graph is perfect if and only if
it contains no odd hole or odd antihole as induced subgraph, where an
\emph{odd hole} is an $n$-cycle with $n$ odd and $n\ge5$, and an
\emph{odd antihole} is the complement of an odd hole.

A \emph{partial preorder} is a binary relation $\to$ on a set $X$ which is
reflexive and transitive. Its \emph{comparability graph} is the graph
on the vertex set $X$ in which $u$ and $v$ are adjacent if and only if either
$u\to v$ or $v\to u$.

\begin{prop}\label{t:ppp}
The comparability graph of a partial preorder is perfect.
\end{prop}

\begin{proof}
In the case of a \emph{partial order}, a partial preorder satifying
\[(x\to y\hbox{ and }y\to x)\Rightarrow x=y,\]
this is one part (the easier part) of \emph{Dilworth's Theorem}~\cite{dilworth}.
For completeness, we sketch the proof.

Let $\Gamma$ be the comparability graph of a partial order. A clique in
$\Gamma$ is a chain in the partial order. If the largest chain has size $k$,
we produce a proper $k$-colouring of the graph as follows: the minimal
elements of the partial order form an independent set, to which we assign
the first colour, and then remove these vertices. Since every maximal chain
contains one of these vertices, the largest chain in the resulting graph has
size $k-1$, and so by induction it can be properly coloured with the
remaining $k-1$ colours.

Now suppose that $\to$ is a partial preorder on $X$. Define a relation
$\equiv$ on $X$ by the rule
\[(x\equiv y)\Leftrightarrow(x\to y\hbox{ and }y\to x).\]
Then $\equiv$ is an equivalence relation. Now define a relation $\le$ on the
set of equivalence classes by the rule that $[x]\le[y]$ if $x'\to y'$ for
some (and hence every) $x'\in[x]$ and $y'\in[y]$; then $\le$ is a partial order.

Now we can refine the partial preorder to a partial order by simply putting
a total order on each equivalence class; this does not change the
comparability graph, which is thus perfect.
\end{proof}

Now we can show that the power graph is a perfect graph. Indeed, we do not
require a group; a semigroup will suffice, or indeed any
\emph{power-associative magma}, that is, a set with a binary operation such
that the associative law holds for the powers of a single element (so that
powers $x^n$ can be unambiguously defined).

\begin{prop}
The directed power graph of any power-associative magma is a partial preorder.
\end{prop}

\begin{proof}
This follows immediately from the facts that $x^1=x$ and $(x^m)^n=x^{mn}$.
\end{proof}

\begin{corollary}\label{c:ppg}
The undirected power graph of a power-associative magma is a perfect graph.\qed
\end{corollary}

\begin{remark}
This result holds also for infinite power-associative magmas, if we say that
an infinite graph is perfect if every finite induced subgraph has clique
number equal to chromatic number.
\end{remark}

\begin{remark}
This satisfactory result shows, for example, that the power graph of a group
cannot contain the $5$-cycle $C_5$ as an induced subgraph (this graph has
clique number $2$ and chromatic number~$3$). In the remainder of this paper
we look at other examples. But first we pose an open problem concerning two
closely related graphs.

The \emph{enhanced power graph} of a group $G$ is the graph with vertex set
$G$, in which $u$ and $v$ are joined if and only if they are both powers of
an element $w$. The \emph{commuting graph} of $G$ is the graph in which $u$
and $v$ are joined if and only if $uv=vu$. It is clear that the power graph
is a spanning subgraph of the enhanced power graph, which is itself a spanning
subgraph of the commuting graph.
\end{remark}

\begin{problem}
\begin{enumerate}
\item For which groups $G$ is the enhanced power graph of $G$ a perfect graph?
\item For which groups $G$ is the commuting graph of $G$ a perfect graph?
\end{enumerate}
\end{problem}

We make two observations about these problems.

\begin{remark}
The paper of Aalipour \emph{et~al.}~\cite{aacns} contains classifications
of groups for which either the enhanced power graph or the commuting graph
is equal to the power graph. These groups are examples satisfying the
conditions of the problem. So the question is: What other groups can occur?
\end{remark}

\begin{remark}
It is not the case that the enhanced power graph or the
commuting graph of every group is perfect. For the commuting graph, we may
take $G=S_5$, the symmetric group of degree~$5$; the five transpositions
\[(1,2), (3,4), (5,1), (2,3), (4,5)\]
have the property that the induced subgraph is $C_5$. For the enhanced
power graph, we can work in a larger symmetric group, and take five cycles of
pairwise coprime lengths with the same intersection pattern as above, using
the fact that two elements of coprime order are joined in the enhanced power
graph if and only if they commute.
\end{remark}

\section{Cographs}

A graph $G$ is a \emph{cograph} if it has no induced subgraph isomorphic to
the four-vertex path $P_4$. Cographs form the smallest class of graphs 
containing the $1$-vertex graph and closed under the operations of disjoint
union and complementation.

Cographs have been rediscovered several times and given several different
names: for example, Sumner~\cite{sumner} called them \emph{heredditary Dacey
graphs}. The class of cographs contains the threshold graphs that we discuss
later, and is included in the class of comparability graphs (they are the
comparability graphs of \emph{N-free partial orders}, those obtained from the
one-point order by the operations of disjoint union and ordered sum). Thus,
cographs are perfect graphs.

In this section, we determine completely the finite nilpotent groups whose
power graph is a cograph. We also give a partial analysis of the question
for more general groups, using the concept of the prime graph (or
Gruenberg--Kegel graph) of a group.

\subsection{Nilpotent groups}

Recall that a finite group $G$ is \emph{nilpotent} if and only if it is a
direct product of its Sylow $p$-subgroups over primes $p$ dividing $|G|$. 
Note that, in a nilpotent group, elements of different prime orders commute.
Our first result determines which finite nilpotent groups $G$ have the property
that their power graphs are cographs.

We begin with:

\begin{lemma}\label{l:pp}
The power graph of a finite group cannot have an induced $4$-vertex path or
cycle in which all four vertices are elements whose orders are powers of the
same prime~$p$.
\end{lemma}

\begin{proof}
Suppose that $(a,b,c,d)$ is such a path or cycle.

In the directed power graph $\vec{P}(G)$, we cannot have $b\to c$. For suppose
that $b\to c$. Then $c$ is a power of $b$, and one of $a$ and $b$ is a power of
the other; so $a,b,c$ are all contained in a cyclic $p$-group. But the power
graph of a cyclic $p$-group is complete, by Proposition~\ref{p:complete}, and
so there is an edge $\{a,c\}$, contrary to assumption.

By the same argument applied to $(b,c,d)$, we cannot have $c\to b$. Therefore
$b$ and $c$ are not joined in the power graph, a contradiction.
\end{proof}  

\begin{theorem}\label{t:nilp_cograph}
Let $G$ be a finite nilpotent group. Then $P(G)$ is a cograph if and only if
either $|G|$ is a prime power, or $G$ is cyclic of order $pq$ for distinct
primes $p$ and $q$.
\end{theorem}

\begin{proof}
Lemma~\ref{l:pp} shows that if $G$ is a $p$-group with $p$ prime, then $P(G)$
contains no induced path of length $3$, and so it is a cograph.

Now we show that $G=C_{pq}$ is a cograph if $p$ and $q$ are distinct primes.
In any group, two elements $a,b$ satisfy $a\to b$ and $b\to a$ if and only
if $\langle a\rangle=\langle b\rangle$; in a cyclic group, this is equivalent
to the condition that $a$ and $b$ have the same order, and hence the same
neighbours. So if $(a,b,c,d)$ is an induced path, then no two of $a,b,c,d$
have the same order. Hence one of them is the identity. But the identity is
joined to all other vertices, a contradiction.

\medskip

Conversely, let $G$ be a finite nilpotent group whose power graph is a cograph.
Suppose first that three primes $p,q,r$ divide $|G|$.
Let $a,b,c$ be elements with orders $p,q,r$ respectively. These three elements
commute pairwise; so $ab$ and $bc$ have orders $pq$ and $qr$ respectively.
Then $(a,ab,b,bc)$ is an induced path, a contradiction.

So suppose that only two primes $p$ and $q$ divide $|G|$. Then $G=P\times Q$
where $P$ is a $p$-group and $Q$ a $q$-group.

Suppose first that $P$ is not cyclic. Then there are elements $a,b\in P$ which
are not adjacent in the power graph. If $c$ is a non-identity element of $Q$,
then $(a,ac,c,bc)$ is an induced path in $P(G)$. So $P$ (and similarly
$Q$) is cyclic.

Now suppose that $|P|>p$. If $a$ is an element of order $p^2$ in $P$, and
$b$ an element of order $q$ in $Q$, then $(a,a^p,a^pb,b)$ is an induced path
in $P(G)$. (These elements have orders $p^2$, $p$, $pq$, $q$ respectively; so
non-consecutive elements are not joined.)

So $|P|=p$, and similarly $|Q|=q$; thus $G=C_p\times C_q=C_{pq}$.
\end{proof}

\subsection{The prime graph, or Gruenberg--Kegel graph}

Suppose that $G$ is a group which contains no element whose order is the
product of two primes. Then every element of $G$ has prime power order, and
every edge of the power graph joins elements of the same prime power order. 
Now Lemma~\ref{l:pp} shows that $P(G)$ contains no induced $P_4$, and so it
is a cograph

These considerations lead us to the notion of the \emph{prime graph} 
$\Pi(G)$ of the group $G$. The vertices of $\Pi(G)$ are the prime divisors
of $|G|$; there is an edge joining $p$ to $q$ if $G$ contains an element
of order $pq$. The argument in the previous paragraph shows:

\begin{theorem}
Let $G$ be a group whose prime graph is a null graph. Then $P(G)$ is a
cograph.
\end{theorem}

The prime graph was introduced by Gruenberg and Kegel in in an unpublished
manuscript studying integral representations of groups in 1975. (The prime
graph is often called the \emph{Gruenberg--Kegel graph}.) They observed
that there are very strong restrictions on the structure of a group whose
prime graph is disconnected. Subsequently, Williams~\cite{williams} (a student
of Gruenberg) published their theorem. With a small addition, it states the
following. Here a $2$-Frobenius group is a group with normal subgroups
$H$ and $K$ with $K\le H$ such that
\begin{enumerate}
\item $H$ is a Frobenius group with Frobenius kernel $K$;
\item $G/K$ is a Frobenius group with Frobenius kernel $H/K$.
\end{enumerate}
The symmetric group $S_4$ is an example, with $H=A_4$ and $K=V_4$.

Moreover, suppose that $G$ has even order. Let $\pi_1$ be the set of primes
which are vertices of the connected component containing the prime $2$. A
$\pi_1$-group denotes a group whose order is divisible only by primes
in~$\pi_1$.

\begin{theorem}
The prime graph of $G$ is disconnected only if $G$ satisfies one of the
following:
\begin{enumerate}
\item $G$ is Frobenius or $2$-Frobenius;
\item $|G|$ is even, and $G$ is an extension of a nilpotent $\pi_1$-group by a
simple group by a $\pi_1$-group.
\end{enumerate}
\end{theorem}

Williams observed that, in the second case of the theorem, the simple
group itself must have disconnected prime graph, and he analysed the
simple groups to find which have disconnected prime graph.  The results
cannot be summarised here since they comprise extended tables. The tables
contained some errors which were corrected by Kondrat'ev and Mazurov~\cite{km}.

The only simple groups with the property that the prime graph is a
null graph are the alternating groups $A_5$ and $A_6$ and the groups
$\mathrm{PSL}(2,7)$, $\mathrm{PSL}(2,8)$, $\mathrm{PSL}(2,17)$,
$\mathrm{PSL}(3,4)$, and $\mathrm{Sz}(8)$. So these groups have power graphs
which are cographs.

In the other direction, Williams observed that, if $G$ is a non-solvable group
with disconnected prime graph and $\pi$ is the vertex set of a connected
component not containing $2$, then $G$ has a nilpotent Hall $\pi$-subgroup
which contains the centraliser of each of its elements. By our earlier
result, we obtain the following.

\begin{theorem}
Let $G$ be a non-solvable group whose power graph is a cograph, and suppose
that the prime graph of $G$ is disconnected. Then, with possibly one
exception, each component of the prime graph is either an isolated vertex or
an isolated edge joining two primes which divide $|G|$ to the first power only
(the exception, if any, is the component containing $2$).
\end{theorem}

Information about the simple groups satisfying this condition can be obtained
from the results of \cite{williams,km}. However, some problems remain. We
give an example. For which prime powers $q$ is power graph of the group
$\mathrm{PSL}(2,q)$ a cograph? If $q$ is even, then the prime graph is the
union of three complete graphs, on the sets of prime divisors of $q$, $q-1$
and $q+1$; if $q$ is odd, we replace $q-1$ and $q+1$ by $(q-1)/2$ and
$(q+1)/2$.

\begin{enumerate}
\item Let $q=2^n$. If $n$ is even, say $n=2m$ with $m>1$, then
$2^{2m}-1=(2^m-1)(2^m+1)$, and the two factors cannot both be primes except
when $m=2$; so $q=4$ and $q=16$ are the only
examples. If $n$ is odd, then we require that $2^n-1$ has at most two prime
divisors; also $3\mid2^n+1$, so we require that $(2^n+1)/3$ is prime. Examples
include $q=8$, $32$, $128$, $2048$, $8192$, \ldots
\item Let $q$ be a power of a prime $p>3$. Then one of $(q-1)/2$ and $(q+1)/2$
is even, and one is divisible by $3$. So there are several possibilities:
\begin{itemize}
\item $(q-1)/2=2r$, $(q+1)/2=3s$, or $(q-1)/2=3s$, $(q+1)/2=2r$, with $r$ and
$s$ primes (possible values of $q$ include $27$, $43$, $67$, \ldots);
\item one of $(q-1)/2$ and $(q+1)/2$ is a power of $2$, the other is three
times a prime (examples $q=31$, $257$, \dots);
\item one of $(q-1)/2$ and $(q+1)/2$ is a power of $3$, the other is twice a
prime (examples $q=19$, $53$, $163$, \dots);
\item one of $(q-1)/2$ and $(q+1)/2$ is a power of $2$, the other a power of
$3$ (by the solution of Catalan's equation, this holds only for $q=5$, $7$
and $17$).
\end{itemize}
\item For $q$ a power of $3$, a similar analysis is possible. Examples 
include $q=9$, $27$, $243$, $2187$, \ldots
\end{enumerate}

\begin{problem}
Are there infinitely many prime powers $q$ such that the power graph of
$\mathrm{PSL}(2,q)$ is a cograph?
\end{problem}

\section{Chordal graphs}

A graph $\Gamma$ is \emph{chordal} if it contains no induced cycles of length
greater than $3$; in other words, every cycle on more than $3$ vertices has
a chord.

Chordal graphs also arise in several different contexts: they are the 
intersection graphs of subtrees of a tree, or the graphs with a perfect
elimination order~\cite{fg}. The class includes the split graphs and is
contained within the class of perfect graphs.

In this section we determine which finite nilpotent groups have chordal
power graphs.
We already know that power graphs contain no odd cycles of length greater
than $3$.

\begin{theorem}
If $G$ is a group of prime power order then $P(G)$ is chordal.
\end{theorem}

\begin{proof}
We saw in Lemma~\ref{l:pp} that $P(G)$ has no induced path of length~$3$,
and hence no induced even cycle of length $6$ or greater. The same lemma
also shows that there is no induced $4$-cycle.
\end{proof}

\begin{lemma}\label{l:p2}
Let $G$ be a group whose order is a power of a prime $p$. Suppose that there
is no induced path of length $2$ in $P(G)$ not containing the identity.
Then either $G$ is cyclic or $G$ has exponent $p$.
Conversely, if $G$ is cyclic or has exponent $p$, then $P(G)\setminus\{1\}$
contains no induced path of length~$2$.
\end{lemma}

\begin{proof}
If $G$ is cyclic then $P(G)$ is complete; if $G$ has exponent $p$, then
$P(G)$ consists of complete graphs of order $p$ with a single vertex $1$ in
common. So the converse holds.

So assume that $G$ is a $p$-group and $P(G)$ has no induced path of length~$2$
not containing the identity; assume for a contradiction that $G$ is neither
cyclic nor of exponent $p$.

First observe that $G$ cannot contain a subgroup $C_{p^2}\times C_p$. For
if $a$ and $b$ were generators of the factors in such a subgroup, then
$(a,a^p,ab)$ would be a path of length~$2$. In particular, the centre of $G$
is elementary abelian. If $Z(G)$ is not cyclic, then choose an element $a$
of order $p^2$ in $G$, and an element $b\in Z(G)\setminus\langle a\rangle$;
then $\langle a,b\rangle\cong C_{p^2}\times C_p$, a contradiction. Thus
$Z(G)\cong C_p$. This argument also shows that, if $a$ is any element of $G$
of order greater than $p$, then some power of $a$ generates $Z(G)$.

Suppose that
$b$ is an element of order $p$ not in $Z(G)$, and $a$ an element of order
$p^2$, so that $a^p\in Z(G)$. If $a$ and $b$ commute, then they generate
$C_{p^2}\times C_p$; so suppose not. Then $(b^{-1}ab)^p=a^p$. If
$b^{-1}ab\notin\langle a\rangle$, then $(a,a^p,b^{-1}ab)$ is a path of
length~$2$; so suppose that $b^{-1}ab\in\langle a\rangle$. Then
$\langle a,b\rangle$ is a non-abelian subgroup of order $p^3$ and exponent
$p^2$, and we may suppose that $b^{-1}ab=a^{p+1}$. Now we can compute that
$(ab)^p=a^p$, and so $(a,a^p,ab)$ is a path of length $2$.
\end{proof}

\begin{theorem}
Let $G$ be a finite nilpotent group which is not of prime power order. Then
$P(G)$ is chordal if and only if $|G|$ has two prime divisors, one
of the two Sylow subgroups is cyclic, and the other has prime exponent.
\end{theorem}

\begin{proof}
If there are three prime divisors of $|G|$, say $p$, $q$, $r$, let $a$, $b$,
$c$ be elements of these orders. Then $(a,ab,b,bc,c,ca,a)$ is an induced
$6$-cycle in $P(G)$.

Suppose that $p$ and $q$ are the only prime divisors of $|G|$, and $P$, $Q$
the corresponding Sylow subgroups. Suppose that $P$ is not cyclic. If also
$Q$ is not cyclic, then choose $a,b$ nonadjacent vertices in $P$ and $c,d$
nonadjacent vertices in $Q$; then $(a,ac,c,bc,b,bd,d,ad,a)$ is an induced
$8$-cycle. So $Q$ is cyclic. If $(a,b,c)$ is an induced path of length~$2$
on non-identity elements of $P$, and $d$ a non-identity element of $Q$, then
$b$ is a power of $a$ and a power of $c$ (but not the reverse) and so
$(b,ad,d,cd)$ is an induced $4$-cycle in $P(G)$. By Lemma~\ref{l:p2}, $P$ has
prime exponent. 

Now $P(C_{p^2}\times C_{q^2})$ is not chordal, since (if the factors are
generated by $a$ and $b$) it contains the cycle
\[((1,b^q),(a^p,b),(a^p,1),(a,b^q)).\]
So $P$ and $Q$ are not both cyclic of composite order; thus one is cyclic and
the other has prime exponent.

Conversely, let $H$ be a group of exponent $p$, and $G=H\times C_{q^\beta}$,
and suppose that $P(G)$ contains a cycle of even length.
Then $P(H)$ consists of complete graphs of size $p$ sharing only the identity.
Any two vertices of the form $(1,y)$ in $G$ are adjacent; so the cycle contains
at most two such vertices, and if two then they must be consecutive.
The remaining vertices all have the form $(h,y)$ with $h\ne 1$. Then
for two consecutive vertices $(h,y)$ and $(h',y')$, we must have
$\langle h\rangle=\langle h'\rangle$, and we may without loss of generality
assume that $h=h'$. So the entire cycle is contained in
$C_p\times C_{q^\beta}$, and it suffices to show that the power graph of 
this group is chordal. The argument above shows that the length of an even
cycle cannot be greater than $4$.

So finally, let $((1,y_1),(1,y_2),(a,y_3),(a,y_4))$ be a $4$-cycle. Let the
order of $y_i$ be $q^{m_i}$. Then we have
\[m_1>m_3\ge m_2>m_4\ge m_1,\]
since, for example, $y_1$ is a power of $y_4$ but is not a power of $y_3$.
But this is impossible.
\end{proof}

\begin{remark}
A group of exponent $2$ is abelian, and so is a direct product of cyclic
groups of order $2$. But there are non-abelian groups of exponent $p$
for any odd prime $p$, for example
\[G=\langle a,b,z:a^p=b^p=z^p=1, [a,b]=z, [a,z]=[b,z]=1\rangle\]
of order $p^3$.
\end{remark}

\section{Threshold graphs and split graphs}

A \emph{threshold graph} is a graph containing no induced subgraph isomorphic
to $P_4$, $C_4$ or $2K_2$ \cite{Chv,Henderson,Mahadev}. Thus every threshold
graph is a cograph. Threshold graphs form the smallest family of graphs
containing the one-vertex graph and closed under the operations of adding
an isolated vertex and adding a vertex joined to all others.

Applications of threshold graphs in computer science and psychology can be
found in \cite{Chv,Henderson}.

An $n$-vertex threshold graph is can be represented by a binary
string $a_1a_2\cdots a_n$ where $a_1=0$ and, for $2\leq a_i\leq n$, $a_i=0$ if
the vertex $i$ is added as an isolated vertex, and $a_i=1$ if it is added as
a vertex joined to all other existing vertices. This sequence gives a simple
representation of the graph.

A graph $G$ is \emph{split} if the vertex set is the disjoint union of two
subsets $A$ and $B$ such that $A$ induces a complete graph and $B$ a null graph.
Split graphs were introduced independently by several authors, first by
F\"oldes and Hammer~\cite{fh}, and are easy to recognise algorithmically
(in particular, they can be recognised by their degree sequences).

A graph is split if and only if it contains no induced subgraph isomorphic
to $C_4$, $C_5$ or $2K_2$.

\medskip

A common feature of these two classes is that both exclude $2K_2$. In \cite{Ma1}, Ma and Feng proved that power graph of a finite group is a split graph if and only is it does not contains $2K_2$. In fact we observe that the same is true for threshold graph.
In this section we are going to determine all finite groups whose power
graph excludes $2K_2$.

We say that the finite group $G$ satisfies the \emph{intersection condition}
(IC) if $G$ does not contain subgroups $H$ and $K$ such that both
$H\setminus K$ and $K\setminus H$ contain elements of order greater than~$2$.

\begin{theorem}
For a finite group $G$, the following conditions are equivalent:
\begin{enumerate}\itemsep0pt
\item $P(G)$ is a threshold graph;
\item $P(G)$ is a split graph;
\item $P(G)$ contains no induced subgraph isomorphic to $2K_2$;
\item $G$ satisfies the intersection condition;
\item $G$ is cyclic of prime power order, or an elementary abelian or dihedral
$2$-group, or cyclic of order $2p$, or dihedral of order $2p^n$ or $4p$, where
$p$ is an odd prime.
\end{enumerate}
\end{theorem}

\begin{proof}
We noted already that each of (a) and (b) implies (c).

Next we show that (c) implies (d). If IC fails, let $H$ and $K$ be subgroups
of $G$ and $x\in H\setminus K$, $y\in K\setminus H$, where $x$ and $y$ are
elements with order greater than $2$. Then $\{x,x^{-1},y,y^{-1}\}$ induces
$2K_2$ (with edges $\{x,x^{-1}\}$ and $\{y,y^{-1}\}$ only).

Now we classify groups satisfying IC. First suppose that $G$ has prime power
order $p^n$. If $p$ is odd, then $G$ has at most one subgroup of each order
$p^i$ with $i<n$, which forces $G$ to be cyclic. So suppose that $p=2$. Let
$H$ be a cyclic subgroup of $G$ of maximum order. If $H=G$, then $G$ is
cyclic; if $|H|=2$, then $G$ is elementary abelian. By IC, every element
outside $H$ has order $2$, which implies that $G$ is dihedral.

Now suppose that $G$ does not have prime power order. By IC, $|G|$ is
divisible by only one odd prime (since two Sylow subgroups of odd prime power
order would violate IC). Let $p$ be an odd prime divisor of $|G|$, and $P$ a
Sylow $p$-subgroup of $G$. Then $P$ is the unique Sylow $p$-subgroup of $G$,
and so is normal. Let $T$ be a Sylow $2$-subgroup. If $T$ contains an element
$x$ of order greater than $2$, then $\langle x\rangle$ and $P$ violate IC; so
$T$ is elementary abelian. Also $P$ is cyclic, and its automorphism group
is therefore cyclic; so $T$ induces an automorphism group of $P$ of order $1$
or $2$, so $|C_T(P)|$ is a subgroup of index at most $2$ in $P$.

If $T$ contained distinct elements $s$ and $t$ centralising $P$, then
$\langle s,P\rangle$ and $\langle t,P\rangle$ violate IC. Also, if
$1\ne t\in C_T(P)$ and $|P|>p$, choose an element $x$ of order $p^2$ in $P$;
then $\langle x\rangle$ and $\langle t,x^p\rangle$ violate IC.

So we are left with three cases:
\begin{itemize}
\item $|P|=2$ and $C_T(P)=1$. Then the non-identity element of $T$ inverts
$P$, and $G$ is dihedral of order $2p^n$, where $|P|=p^n$.
\item $|P|=2$ and $C_T(P)=P$. Then $|P|=p$ and $G$ is cyclic of order $2p$.
\item $|P|=4$ and $|C_T(P)|=2$. Then $G=C_2\times D_p=D_{2p}$.
\end{itemize}

The final step is to show that the groups $G$ in (e) have the property that
$P(G)$ is threshold and split, that is, contains no induced $P_4$, $C_4$, 
$2K_2$ or $C_5$. Here $C_5$ is excluded since $P(G)$ is perfect
(Corollary~\ref{c:ppg}). We note also that, if $H$ is a graph having no
isolated vertex and no vertex joined to all others, and $P(C_n)$ is $H$-free,
then also $P(D_n)$ is $H$-free. For $P(D_n)$ consists of $P(C_n)$ with $n$
further vertices joioned only to the identity (which is joined to all vertices
of $P(C_n)$), so any induced copy of $H$ would have to be contained in
$P(C_n)\setminus\{1\}$.

Thus we only have to deal with cyclic groups. Now if $G$ is cyclic of prime
power order, then $P(G)$ is complete; and if $G$ is cyclic of order $2p$,
then the complement of $P(G)$ is a star (on the elements of orders $2$ and $p$)
together with $p$ isolated vertices (the identity and the generators). All
these graphs exclude $P_4$, $C_4$ and $2K_2$ (a self-complementary collection
of graphs).
\end{proof}

\subsection*{Acknowledgements}
The author Pallabi Manna is supported by CSIR (Grant No-09/983(0037)/2019-EMR-I). Ranjit Mehatari thanks the SERB, India,
for financial support (File Number: CRG/2020/000447) through the Core
Research Grant.

\end{document}